\documentclass[12pt,a4paper]{article}
\usepackage[T1,T2A]{fontenc}
\usepackage[cp1251]{inputenc}
\usepackage{mathrsfs}
\usepackage{amsmath}
\usepackage{amsthm}
\usepackage{amssymb}
\usepackage{amsfonts}
\usepackage[normalem]{ulem}
\usepackage{hyperref}

\makeatletter

\newtheorem{theorem}{Theorem}

\newtheorem{lemma}{Lemma}

\theoremstyle{definition}
\newtheorem{remark}{Remark}

\newcommand{\V}{\mathscr{V}}

\makeatother

\begin{document}
\title{\textbf{Normal Approximation for $U$- and $V$-statistics of a Stationary Absolutely Regular Sequence}}
\author{Vladimir G. Mikhailov\thanks{Steklov Mathematical Institute of Russian Academy of Sciences, Moscow, Russia}\\\vspace{-0.15cm}\footnotesize{\it{mikhail@mi-ras.ru}}\and Natalia M. Mezhennaya\thanks{Bauman Moscow State Technical University, Moscow, Russia}\\\vspace{-0.15cm}\footnotesize{\it{natalia.mezhennaya@gmail.com}}}
\date{}
\maketitle

\begin{abstract}
Let $(X_{n,t})_{t=1}^{\infty}$ be a stationary absolutely regular sequence of real random variables with the distribution dependent on the number~$n$. The paper presents sufficient conditions for the asymptotic normality (for $n\to\infty$ and common centering and normalization) of the distribution of the nonhomogeneous $U$-statistic of order $r$ which is given on the sequence $X_{n,1},\ldots,X_{n,n}$ with a kernel also dependent on $n$. The same results for $V$-statistics also hold.
To analyze sums of dependent random variables with rare strong dependencies, the proof uses the approach that was proposed by S.~Janson in 1988 and upgraded
by V.~Mikhailov in 1991 and M.~Tikhomirova and V.~Chistyakov in 2015.
\end{abstract}



\section*{Introduction}

The study of a special class of functionals of a sequence of random variables
$X_1,\ldots,X_n$ of the form
\begin{equation}\label{U1}
\hat U_n=\frac{1}{C_n^m}\sum_{1\le i_1 <\ldots < i_m\le n}f(X_{i_1},\ldots,X_{i_m}), \end{equation}
which were called {\it $U$-statistics}, began in the middle of the last century due to
the investigation of the properties of sample characteristics (see~\cite{Hoeff} and the
bibliography therein). The number $m$ is the {\it order} and symmetrical function $f(x_1,\ldots, x_m)$ is the {\it kernel} of $U$-statistic. Examples of $U$-statistics are sample moments, Gini's mean difference, Spearman's rank correlation, etc.

It is known that such variables as the number of repetitions and the number of repetitions of tuples~\cite{ZM74}--\cite{T09}, the number of pairs of $H$-equivalent tuples~\cite{Mikh02}--\cite{Sh08}, etc., in the random discrete sequence $X_1,\ldots,X_n$ belong (up to of the factor before the sum) to a class of quantities of the form~\eqref{U1}.

In the following decades, a large number of research papers appeared devoted to the asymptotic properties of $U$-statistics of sequences of independent
identically distributed random variables (see, e.g., the bibliography in~\cite{KorBor}).
In proving asymptotic normality in the case of increasing sums, W. Hoeffding~\cite{Hoeff} proposed a method for approximating the distribution of $U$-statistic
by distribution of the sum of specially constructed independent random variables. This approach, in different forms, is also used to study $U$-statistics of
sequences of random variables with conditions of weak or other dependence in the scheme of increasing sums (see, e.g.,~\cite{Yosh}--\cite{Dehling}).

The idea of the results of these papers is given by the following theorem of K.~Yoshihara ~\cite{Yosh} which we present in a simplified form.

Let $(X_{n})_{n=1}^{\infty}$ for every $n=1,2,\ldots$ be a strictly stationary sequence of real random variables satisfying the {\it absolute regularity condition}~\cite[p.\,3]{Dou}
\begin{equation}\label{eq:AR}
\beta(n)=\mathbf{E}\left\{\sup_{A\in {\cal F}_n^\infty}\left|\mathbf{P}\{A|{\cal F}_{-\infty}^0\}-\mathbf{P}\{A\}\right|\right\}\downarrow 0, \quad
n\to\infty,
\end{equation}
where ${\cal F}_a^b$ is the $\sigma$-algebra of events generated by the random variables $X_a,\ldots,X_b$.

Let
$$
\theta =
\mathbf{E}f(\tilde X_{1},\ldots,\tilde X_{m}), \quad f_1(x)= \mathbf{E}f(x, \tilde X_{2},\ldots,\tilde X_{m}),
$$
where $\tilde X_{1},\ldots,\tilde X_{m}$  are independent copies of $X_{1}$,
$$
\sigma^2 = \left(\mathbf{E}f_1^2(X_1)-\theta^2\right) + 2\sum_{t=1}^\infty \left(\mathbf{E}(f_1(X_1)f_1(X_{t+1}))-\theta^2\right).
$$

\begin{theorem}[Theorem 1 of~\cite{Yosh}]
Let $n\to\infty$ and there is a number $\delta>0$ such that
\begin{gather*}
\mathbf{E}\left|f(\tilde X_{1},\ldots,\tilde X_{m})\right|^{2+\delta} \le M_0<\infty, \label{eq3}
\\
\mathbf{E}\left|f(X_{i_1},\ldots,X_{i_m})\right|^{2+\delta} \le M_0<\infty \quad \forall \ 1\le i_1 <\ldots i_m \le n, \label{eq4}
\\
\beta(n)=O\left(n^{-(2+\delta')/\delta'}\right) \quad \mbox{for \ some} \quad \delta', \ 0<\delta'<\delta. \label{eq5}
\end{gather*}
Then, if $\sigma^2>0$ holds,
the distribution function of the random variable $\frac{\sqrt{n}}{m\sigma}(\hat U_n-\theta)$ converges to the distribution function of the standard normal law.
\end{theorem}

The research paper~\cite{Hashim93}(see also~\cite{Hashim87,Hashim92}) was devoted to adaptation the results of K.~Yoshihara to triangular array schemes.
Sh.~Khashimov in~\cite{Hashim93} considered the case of second-order $U$-statistic whose kernel $f(x_1, x_2) = f_n(x_1, x_2)$ can change for $n\to \infty$. Again, the method of W.~Hoeffding~\cite{Hoeff} was used.

The method of moments was no less promising for studying $U$-statistics in the triangular array scheme for dependent random variables.
Back in 1975, V.~Mikhailov~\cite{Mikh75}, using the direct application of this method, derived sufficient conditions for asymptotic normality for a special case of $U$-statistics of a sequence of finitely dependent random variables in a triangular array scheme (let's call it \textit{the wide triangular array scheme}), where for $n\to \infty$ changes are allowed both to the kernel $f_n(x_1,\ldots,x_m)$ and the distribution of the sequence  $X_{n,1},X_{n,2},\ldots$ (now in the notation we have to indicate dependence of the kernel and distribution on $n$).

A modern variation of the method of moments which was proposed by Svante Janson~\cite{Janson Svante,M91} and upgraded by V.~Mikhailov in~\cite{M91} and  by M.~Tikhomirova and V.~Chistyakov in~\cite{TihChis} allows to obtain simpler and substantially more general sufficient conditions for the asymptotic normality of $U$- and $V$-statistics of any order of a sequence of random variables satisfying the absolute regularity condition  in the wide triangular array scheme which we present in this paper. These results complement the results of K.~Yoshihara~\cite{Yosh} and V.~Mikhailov~\cite{Mikh75}.

It also should be noted that for problems related to tuples in a discrete random sequence (see, e.g., \cite{MSh03,Sh05,MSh14,Mikh15}), the present results allow to consider the case of simultaneous consistent growth of the length of the random sequence $n$ and the length of the tuple $s$ to infinity. A separate work is supposed to be devoted to these applications.

\section{Limit Theorems}

Let $(X_{n,t})_{t=1}^{\infty}$ for every $n=1,2,\ldots$ be a strictly stationary sequence of real random variables (e.g., the joint probability distribution function of $(X_{n,t_k+\tau})_{k=1}^{m}$ is equal to the joint  probability distribution function of $(X_{n,t_k})_{k=1}^{m}$ for all $\tau,t_1,\ldots,t_m=1,2,\ldots,\infty$  and all $m\ge 1$)
satisfying the absolute regularity condition~\eqref{eq:AR}. 

Let $f_{n;j_1,\ldots,j_r}: \mathbb{R}^r\to \mathbb{R}$ be a bounded measurable function for every $n\ge 1$ and $1\le j_1<\ldots<j_r\le n:$
$$|f_{n;j_1,\ldots,j_r}(x_{j_1},\ldots,x_{j_r})|\le F_n<\infty.$$

The functionals called the {\it nonhomogeneous $U$-statistic} and {\it $V$-statistic} with the kernel $f_{n;j_1,\ldots,j_r}$ are given by the formulas (the definitions are given in~\cite{Hoeff} or~\cite{KorBor}):
\begin{gather}
U_n=U_n(X_{n,1},\ldots,X_{n,n})=\sum_{1\le j_1<\ldots<j_r\le n}f_{n;j_1,\ldots,j_r}(X_{n,j_1},\ldots,X_{n,j_r}),\label{U}
\\
V_n=V_n(X_{n,1},\ldots,X_{n,n})=\sum_{ j_1,\ldots,j_r=1}^{n} f_{n;j_1,\ldots,j_r}(X_{n,j_1},\ldots,X_{n,j_r}),\label{V}
\end{gather}
respectively (in contrast to the traditional definition~\eqref{U1}, the factors $1/C_n^r$ and $1/n^r$ are omitted before the sums).

Let  
$$U_n^*=\frac{U_n-\mathbf{E}U_n}{\sqrt{\mathbf{D}U_n}},\quad V_n^*=\frac{V_n-\mathbf{E}V_n}{\sqrt{\mathbf{D}V_n}}.$$

\begin{theorem}\label{th1}
Let $(X_{n,t})_{t=1}^{\infty}$ for every $n=1,2,\ldots$ be a strictly stationary sequence of real random variables satisfying the absolute regularity condition~\eqref{eq:AR}.
For $n\rightarrow\infty,$ let the distribution of $(X_{n,t})_{t=1}^{\infty}$, the measurable $($for every $n\ge 1)$ function family $\left\{f_{n;j_1,\ldots,j_r},1\le j_1<\ldots<j_r\le n\right\}$, the number $m_n$ and the other parameters marked by index $n$  vary so that  $b_{0}\in(0,2/3]$ exists such that for every natural number 
$R$ and all $b\in(0,b_{0}]$
\begin{equation}\label{eq:CondTheorem2}
\frac{F_n^2 m_n^{2-b}n^{2(r-1)+b}r^{4-2b}}{\mathbf{D}U_{n}}+\bigl(\beta_n(m_n)\bigr)^{b}\frac{F_n^{2} n^{2r}}{\mathbf{D}U_{n}}\rightarrow 0.
\end{equation}
Then the moments and distribution function of the random variable $U_n^{*}$  converge to the moments and distribution function of the standard normal law. \rm
\end{theorem}

A similar statement holds for $V$-statistics.

\begin{theorem}\label{th1a}
Let $(X_{n,t})_{t=1}^{\infty}$ for every $n=1,2,\ldots$ be a strictly stationary sequence of real random variables satisfying the absolute regularity condition~\eqref{eq:AR}.
For $n\rightarrow\infty,$ let the distribution of $(X_{n,t})_{t=1}^{\infty}$, the measurable $($for every $n\ge 1)$ function family $\left\{f_{n;j_1,\ldots,j_r},1\le j_1<\ldots<j_r\le n\right\}$, the number $m_n$ and the other parameters marked by index $n$  vary so that  $b_{0}\in(0,2/3]$ exists such that for every natural number 
$R$ and all $b\in(0,b_{0}]$
\begin{equation}
\frac{F_n^2 m_n^{2-b}n^{2(r-1)+b}r^{4-2b}}{\mathbf{D}V_{n}}+\bigl(\beta_n(m_n)\bigr)^{b}\frac{F_n^{2} n^{2r}}{\mathbf{D}V_{n}}\rightarrow 0.
\end{equation}
Then the moments and distribution function of the random variable $V_n^{*}$  converge to the moments and distribution function of the standard normal law. \rm
\end{theorem}

We consider a special case in which the absolute regularity coefficient~\eqref{eq:AR} decreases faster than any degree of $t$ for $t\to\infty$.

\begin{theorem}\label{th2}
Let $(X_{n,t})_{t=1}^{\infty}$ for every $n=1,2,\ldots$ be a strictly stationary sequence of real random variables satisfying the absolute regularity condition~\eqref{eq:AR}.
Let the number $F_n$ be fixed starting at some value of  $n,$  $n\to\infty,$  and the joint distribution of the random variables 
$X_{n,1},\ldots,X_{n,n}$ and the measurable function $|f_{n;j_1,\ldots,j_r}|<F_n$ vary so that
$$
\beta_n(t)\le t^{-h(t)}, \quad  \mathbf{D}U_n\ge Cn^{2(r-1)+\varkappa}, \quad C, \varkappa>0,
$$
where the positive function $h(t)\to\infty$ for $t\to\infty$. Then the moments and distribution function of the random variable $U_n^{*}$ 
converge to the moments and distribution function of the standard normal law. \rm
\end{theorem}

\begin{theorem}\label{th2a}
Let $(X_{n,t})_{t=1}^{\infty}$ for every $n=1,2,\ldots$ be a strictly stationary sequence of real random variables satisfying the absolute regularity condition~\eqref{eq:AR}.
Let the number $F_n$ be fixed starting at some value of  $n,$  $n\to\infty,$  and the joint distribution of the random variables 
$X_{n,1},\ldots,X_{n,n}$ and the measurable function $|f_{n;j_1,\ldots,j_r}|<F_n$ vary so that
$$
\beta_n(t)\le t^{-h(t)}, \quad  \mathbf{D}V_n\ge Cn^{2(r-1)+\varkappa}, \quad C, \varkappa>0,
$$
where the positive function $h(t)\to\infty$ for $t\to\infty$. Then the moments and distribution function of the random variable $V_n^*$
converge to the moments and distribution function of the standard normal law. \rm
\end{theorem}

\section{On the Method of Proving of Limit Theorems}


In 1988, Svante Janson~\cite{Janson Svante} proposed a simple technique for deriving sufficient conditions for the asymptotic normality of bounded random variables $Y_1,\ldots,Y_T$ with a joint distribution described by the {\it dependency graph}.
Only one vertex in the dependency graph corresponds to each random variable $ Y_i $, and these vertices are connected by a set of edges. The following condition is satisfied:

if $\V_1$ and $\V_2$ are two disjoint subsets of graph vertices such that no edge of the graph has one endpoint in $V_1$ and the other in $V_2$, then the sets of random variables $\{Y_i,i\in \V_1\}$ and $\{Y_i,i\in \V_2\}$ are independent.

In~\cite{M91}, this approach was presented more generally and was subsequently used in numerous research studies devoted to the study of asymptotic distributions of  functionals depending on a sequence of independent random variables (\cite{T08,T09,Sh07,Sh08}). Finally, in 2015, M.~Tikhomirova and V.~Chistyakov~\cite{TihChis} proposed a modification of the method~\cite{Janson Svante} and~\cite{M91} which is applicable to the families of random variables with a complete dependency graph, but the majority of dependencies between the variables are weak.

\begin{remark}
A year after~\cite{TihChis}, the paper~\cite{Fer1} appeared on the site arXiv.org and was devoted to transferring the approach by~\cite{Janson Svante} and ~\cite{M91} to the case in which the dependency graph is a complete graph, but the majority of dependencies are weak. In~\cite{Fer1}, the joint distribution of a set of quantities is described by a weighted dependency graph in which each edge is assigned a numerical characteristic (weight), which describes the degree of dependence between adjacent variables  in a certain way. The form of asymptotic normality conditions in~\cite{Fer1} resembles similar conditions of~\cite{M91}, but the values included in the conditions are now determined by the weighted dependency graph. The results of applying the conditions by~\cite{Fer1} to specific problems are presented in~\cite{Fer1} and~\cite{Fer2}.
\end{remark}

We present the main result of M.~Tikhomirova and V.~Chistyakov~\cite{TihChis}. We assume that the joint distribution of the variables $Y_{1},\ldots,Y_T$, $|Y_i|\le F<\infty$, is determined by the {\it characterizing} graph $\varGamma $. This is an undirected graph with the  set of vertices $\V=\{1,\ldots,T\}$ and the following properties:

1) if the random variables $ Y_i $ and $ Y_j $ are dependent, then the vertices $i$ and $j$ are connected by an edge (in particular, the graph $\varGamma$ contains loops at all vertices);

2) for any natural number $ R,$ any subset $\V'\subset \V$, $|\V'|\le R$, and any of its partition $\{\V_{1},\V_{2}\}$ (i.e., $\V_{1}\cap \V_{2}=\varnothing$, $\V_{1}\cup \V_{2}=\V'$) such that there are no edges connecting vertices from $\V_{1}$ with vertices from $\V_{2}$, there is a number $\gamma_R\in (0,1)$ such that
\begin{equation}\label{eq:Chis1}
\left|\mathbf{E}\Bigl(\prod_{t\in \V'}Y_{t}\Bigr)-\mathbf{E}\Bigl(\prod_{t\in \V_{1}}Y_{t}\Bigr)\mathbf{E}\Bigl(\prod_{t\in \V_{2}}Y_{t}\Bigr)\right|\leq\gamma_R F^{|\V'|}.
\end{equation}

For any subset $\V'\subset \V,$ we define its set of strong dependencies $L(\V')$ as the set of those vertices of $\varGamma$ that are connected by the edges with vertices from the set $\V'$. Let $\mathscr{\mathcal{F}}(\V')=\sigma\{Y_{t},t\in \V'\}$
be $\sigma$-algebra generated by the random variables $\{Y_{t},t\in \V'\},$
$M=\sum_{t=1}^{T}\mathbf{E}|Y_{t}|$, and
\begin{equation}\label{eq:QR-def}
Q_R=\underset{\V'\subset \V:|\V'|\leq R}{\max}\sum_{t\in L(\V')}\mathbf{E}\left(|Y_{t}|\bigl|\mathcal{F}(\V')\right),\quad R=1,2,\ldots.
\end{equation}

We put
$$
S_{T}=\sum_{t=1}^{T}Y_{t},\quad S_{T}^{*}=(S_{T}-\mathbf{E}S_{T})/\sqrt{\mathbf{D}S_{T}}.
$$

We suppose that the joint distribution of the random variables $Y_{1},\ldots,Y_{T}$ depends on the natural number $n$ assumed as a parameter. All characteristics mentioned above also depend on $n$:
$T=T_n$, $F=F_n$, $\gamma_R=\gamma_{R,n}$, $M=M_n$, $Q_R=Q_{R,n},$ etc.

\begin{theorem}[Theorem~1 from \cite{TihChis}]\label{TC1}
For $n\rightarrow\infty,$ let the numbers $T_{n}\rightarrow\infty,$ the joint distribution of the random variables $Y_{1},\ldots,Y_{T_{n}},$ and the other parameters marked by index $n$ vary so that $b_{0}\in(0,2/3]$ exists such that for all $b\in(0,b_{0}]$ and any natural number $R$
\begin{equation}\label{eq:CondTh1}
\frac{M_{n}^{b}(Q_{R,n})^{2-b}}{\mathbf{D}S_{T_n}}+\gamma_{R,n}^{b}\frac{(F_{n}T_{n})^{2}}{\mathbf{D}S_{T_n}}\rightarrow 0.
\end{equation}
Then the moments and distribution function of the random variable $S_{T_{n}}^{*}$ converge to the moments and distribution function of the standard normal law.
\end{theorem}

\section{Proofs of Theorems~\ref{th1} and~\ref{th2} for $U$-statistics}

The number of summands $w_n(j_1,\ldots,j_r)=f_{n;j_1,\ldots,j_r}(X_{n,j_1},\ldots,X_{n,j_r})$ in~\eqref{U} is $T_n=C_n^r$ (we recall that $C_n^r$ denotes  Binomial coefficient). We construct a characterizing graph for the family of random variables $w_n(j_1,\ldots,j_r)$ as follows. We define some positive integer number $m$ and denote the characterizing graph by $\varGamma_{n,m}$.

Each variable $w_n(j_1,\ldots,j_r)$ is one-to-one assigned to the vertex $\alpha=(j_1,\ldots,j_r)$ in the graph $\varGamma_{n,m}$.
We denote the set of vertices by $\V(\varGamma_{n,m})$. The total number of vertices in the graph $\varGamma_{n,m}$ is $T_n=C_n^r$.

The  set of edges is defined as follows:

1) the graph $\varGamma_{n,m}$ contains the loop at every vertex;

2) the vertices $\alpha=(j_1,\ldots,j_r)$ and $\tilde{\alpha}=(\tilde{j}_1,\ldots,\tilde{j}_r)$, $\alpha,\tilde{\alpha}\in \V(\varGamma_{n, m}),\alpha\neq \tilde{\alpha},$ are connected by the edge from $\varGamma_{n, m}$ if and only if at least one of inequalities holds:
\begin{equation}\label{eq:CondV}
|j_k-\tilde{j}_l|\le m,\quad k,l=1,\ldots,r.
\end{equation}

We show that $\varGamma_{n,m}$ satisfies the property~\eqref{eq:Chis1}. 

\begin{lemma}\label{l1}
For every $n,$  let $(X_{n,t})_{t=1}^{\infty}$ be a strictly stationary sequence of real random variables satisfying the absolute regularity condition~\eqref{eq:AR}. Then for any set $\V'\subseteq \V(\varGamma_{n, m})$ and  its partitions $\V_{1}\cup \V_{2}=\V'$ such that $\V_{1}\cap \V_{2}=\varnothing,$ where there are no edges with one endpoint in $\V_{1}$ and the other endpoint in~$\V_{2},$
\begin{gather}
\left|\mathbf{E}\Bigl(\prod_{\alpha\in \V'}f_{n;j_1,\ldots,j_r}(X_{n,j_1},\ldots,X_{n,j_r})\Bigr)\right.-\notag
\\
-\left.\mathbf{E}\Bigl(\prod_{\alpha\in \V_{1}}f_{n;j_1,\ldots,j_r}(X_{n,j_1},\ldots,X_{n,j_r})\Bigr)\mathbf{E}\Bigl(\prod_{\alpha\in \V_{2}}f_{n;j_1,\ldots,j_r}(X_{n,j_1},\ldots,X_{n,j_r})\Bigr)\right|\le\notag
\\
\label{Lemma1}
\le 8|V'|F_n^{|\V'|}\beta_n(m).
\end{gather} 
\end{lemma}

The inequality~\eqref{Lemma1} shows that the condition~\eqref{eq:Chis1} is satisfied for the graph $\varGamma_{n,m}$ for every $n$ and
\begin{equation}\label{gamma}
\gamma_R=\gamma_{R,n,m}=8 R \beta_n(m).
\end{equation}


We need the following statement to derive Lemma~\ref{l1}.

Let $I=(i_1,\ldots,i_{|I|})$, $1\le i_1<\ldots < i_{|I|}\le n$,  and $I'=(i'_1,\ldots,i'_{|I'|})$, $1\le i'_1<\ldots < i'_{|I'|}\le n$, be the sets of natural numbers, and $g_1$ and $g_2$ be measurable functions of $|I|$ and $|I'|$ real variables, respectively.

\begin{lemma}\label{l2}
Let the strictly stationary sequence of random variables $(X_t)_{t=1}^{\infty}$ satisfy the absolute regularity condition~\eqref{eq:AR} with the coefficient $\beta(m)$, $|g_1(X_i; i\in I)|\le F'$, $|g_2(X_{i'}; i'\in I')|\le F'',$ and
\begin{equation}\label{Lemma2*cond}
\min_{i\in I,\, i'\in I'}\{|i-i'|> m\}.
\end{equation}
Then
\begin{gather}
\left|\mathbf{E}\bigl\{g_1(X_i; i\in I)g_2(X_{i'}; i'\in I')\bigr\}-\mathbf{E}g_1(X_i; i\in I)\mathbf{E}g_2(X_{i'}; i'\in I')\right| \le\notag
\\
\le 8(|I|+|I'|)F' F'' \beta(m).\label{Lemma2*}
\end{gather} 
\end{lemma}

\begin{proof}[Proof of Lemma~$\ref{l2}$] The sets $I$ and $I'$ can be one-to-one split into disjoint subsets $I=I_1\cup\ldots\cup I_{s_1}$ and $I'=I'_1\cup\ldots\cup I'_{s_2}$, where $1\le s_1,s_2\le |I|+|I'|$, $|s_1-s_2|\in \{0,1\}$, satisfying one of the following two conditions:
\begin{gather}
\max_{i'\in I'_k}{i'}<\min_{i\in I_k}{i}, \quad \max_{i\in I_k}{i}<\min_{i'\in I'_{k+1}}{i'}, \quad k=1,\ldots,\max\{s_1,s_2\}-1,\label{order1*}
\\
\max_{i\in I_k}{i}<\min_{i'\in I'_k}{i'}, \quad \max_{i'\in I'_k}{i'}<\min_{i\in I_{k+1}}{i}, \quad k=1,\ldots,\max\{s_1,s_2\}-1.\label{order2*}
\end{gather}
We put
\begin{gather*}
X(I_k)=\left(X_i;\,i\in I_k\right), \quad k=1,\ldots, s_1,
\\
 X(I'_k)=(X_{i'};\,i'\in I'_k), \quad k=1,\ldots, s_2.
\end{gather*}
Then, we can write
\begin{gather*}
g_1(X_i; i\in I)=g_1\left(X(I_1),\ldots,X(I_{s_1})\right),
\\
g_2(X_{i'}; i'\in I')=g_2\left(X(I'_1),\ldots,X(I'_{s_2})\right).
\end{gather*}

According to the order~\eqref{order1*},~\eqref{order2*}, we denote $I_k$ and $I'_k$ by combined notation  $H_k$. In the first case $H_1=I'_1,H_2=I_1,H_3=I'_2,\ldots$, and in the second case $H_1=I_1,H_2=I'_1,H_3=I_2,\ldots.$
Therefore, we use the notation 
$X(H_k)=(X_i;\, i\in H_k)$, $k=1,\ldots,s_1+s_2$. It follows from the definitions that 
\begin{equation}\label{order3}
\min_{i\in H_{k+1}}{i}-\max_{i\in H_k}{i}>m, \quad k=1,\ldots,s_1+s_2-1.
\end{equation}

Let $\tilde X(I_k)$,
$\tilde X(I'_k)$, $\tilde X(H_k)$ be mutually independent copies of the random variables $X(I_k)$, $k=1,\ldots, s_1$, $X(I'_k)$, $k=1,\ldots, s_2$, $X(H_k)$, $k=1,\ldots, s_1+s_2$.

We use the notation
\begin{equation}\label{h}
h(x_i; i\in I\cup I')= g_1(x_i; i\in I)g_2(x_{i'};i'\in I')
\end{equation}
and
\begin{equation}\label{h_sigma1}
h_{\sigma}(x_i; i\in I\cup I') = h(x_{\sigma_i}; i\in I\cup I'),
\end{equation}
where $\sigma=(\sigma_1,\ldots,\sigma_{s_1+s_2})$ is a permutation of the numbers $1,2,\ldots, s_1+s_2$ such that
\begin{equation}\label{h_sigma2}
h(X_i: i\in I\cup I')=h_{\sigma}\bigl(X(H_1),\ldots,X(H_{s_1+s_2})\bigr).
\end{equation}

We use one result of K.~Yoshihara (see~\cite{Yosh}, Lemma 1). It refers to a more general case and, in particular, instead of the boundedness of the function $h,$ it requires the property $\mathbf{E}|h(x_i; i\in I\cup I')|^{1+\delta} \le B_\delta<\infty$ for some $\delta>0$.

In our case, $B_\delta \le (F'F'')^{1+\delta}<\infty$ for every $\delta>0$. Thus, Yoshihara's lemma gives that
\begin{gather*}
\bigl|\mathbf{E}h_{\sigma}\bigl(X(H_1),X(H_2),\ldots,X(H_{s_1+s_2})\bigr)-
\\
-\mathbf{E}h_{\sigma}\bigl(\tilde X(H_1),X(H_2),\ldots,X(H_{s_1+s_2})\bigr)\bigr|\le
\\
\le   \inf_{\delta>0}{\left\{4B_\delta^{1/(1+\delta)}(\beta(m))^{\delta/(1+\delta)}\right\}}\le   4F' F''\beta(m).
\end{gather*}
Analogously, we derive the inequality 
\begin{gather*}
\bigl|\mathbf{E}h_{\sigma}\bigl(x(H_1),X(H_2),X(H_3),\ldots,X(H_{s_1+s_2})\bigr)-
\\
-\mathbf{E}h_{\sigma}\bigl(x(H_1),\tilde X(H_2),X(H_3),\ldots,X(H_{s_1+s_2})\bigr)\bigr|\le 4F' F''\beta(m).
\end{gather*}
for the function $h_{\sigma}\bigl(x(H_1),X(H_2),X(H_3),\ldots,X(H_{s_1+s_2})\bigr)$, where $x(H_1)= (x_i; i\in H_1)$.
It follows that
\begin{gather*}
\bigl|\mathbf{E}h_{\sigma}\bigl(\tilde X(H_1),X(H_2),X(H_3),\ldots,X(H_{s_1+s_2})\bigr)-
\\
-\mathbf{E}h_{\sigma}\bigl(\tilde X(H_1),\tilde X(H_2),X(H_3),\ldots,X(H_{s_1+s_2})\bigr)\bigr|\le 4 F' F'' \beta(m).
\end{gather*}
We carry out similar estimates for other differences. The last among them is
\begin{gather*}
\bigl|\mathbf{E}h_{\sigma}\bigl(\tilde X(H_1),\tilde X(H_2),\ldots,\tilde X(H_{s_1+s_2-1}),X(H_{s_1+s_2})\bigr)-
\\
-\mathbf{E}h_{\sigma}\bigl(\tilde X(H_1),\tilde X(H_2),\ldots,\tilde X(H_{s_1+s_2-1}),\tilde X(H_{s_1+s_2})\bigr)\bigr|\le 4F' F''\beta(m).
\end{gather*}
Summarizing the obtained estimates and using the triangle inequality, we obtain
\begin{gather*}
\bigl|\mathbf{E}h_{\sigma}\bigl(X(H_1),X(H_2),\ldots,X(H_{s_1+s_2})\bigr)-
\\
-\mathbf{E}h_{\sigma}\bigl(\tilde X(H_1),\tilde X(H_2),\ldots,\tilde X(H_{s_1+s_2})\bigr)\bigr|\le
\end{gather*}
\begin{equation}\label{eq100}
\le  4(s_1+s_2) F' F'' \beta(m)= 4(|I|+|I'|) F' F''\beta(m).
\end{equation}

It follows from~\eqref{h},  \eqref{h_sigma1}, \eqref{h_sigma2}, \eqref{eq100} that
\begin{gather}
\bigl|\mathbf{E}\bigl\{g_1\bigl(X(I_1),X(I_2),\ldots,X(I_{s_1})\bigr)g_2\bigl(X(I'_1),X(I'_2),\ldots,X(I'_{s_2})\bigr)\bigr\}-\notag
\\
-\mathbf{E}\bigl\{g_1\bigl(\tilde X(I_1),\tilde X(I_2),\ldots, \tilde X(I_{s_1})\bigr)g_2\bigl(\tilde X(I'_1), \tilde X(I'_2),\ldots, \tilde X(I'_{s_2})\bigr)\bigr\}\bigr|\le \notag
\\
\le 4(|I|+|I'|) F' F''\beta(m).\label{g1g2}
\end{gather}

Analogously~\eqref{eq100}, we derive the inequalities
\begin{gather*}
\left|\mathbf{E}g_1\bigl(X(I_1),\ldots, X(I_{s_1})\bigr)-\mathbf{E}g_1\bigl(\tilde X(I_1),\ldots,\tilde X(I_{s_1})\bigr)\right|\le
 4|I| F'\beta(m),
\\
\left|\mathbf{E}g_2\bigl(X(I'_1),\ldots, X(I'_{s_2})\bigr)-\mathbf{E}g_2\bigl(\tilde X(I'_1),\ldots,\tilde X(I'_{s_2})\bigr)\right|\le
  4|I'| F'' \beta(m).
\end{gather*}
It follows from these inequalities, the inequalities
$$
\mathbf{E}\bigl|g_1\bigl(X(I_1),\ldots, X(I_{s_1})\bigr)\bigr| \le F', \quad \mathbf{E}\bigl|g_2\bigl(\tilde X(I'_1),\ldots, \tilde X(I'_{s_2})\bigr)\bigr| \le F''
$$
and the triangle inequality that
\begin{gather}
\Bigl|\mathbf{E}g_1\bigl(X(I_1),\ldots, X(I_{s_1})\bigr)\mathbf{E}g_2\bigl(X(I'_1),\ldots, X(I'_{s_2})\bigr) -\notag
\\
- \mathbf{E}g_1\bigl(\tilde X(I_1),\ldots,\tilde X(I_{s_1})\bigr)\mathbf{E}g_2\bigl(\tilde X(I'_1),\ldots,\tilde X(I'_{s_2})\bigr)\Bigr|\le \notag
\\
\le \mathbf{E}\bigl|g_1\bigl(X(I_1),\ldots, X(I_{s_1})\bigr)\bigr|\Bigl|\mathbf{E}g_2\bigl(X(I'_1),\ldots, X(I'_{s_2})\bigr) - \mathbf{E}g_2\bigl(\tilde X(I'_1),\ldots,\tilde X(I'_{s_2})\bigr)\Bigr|+\notag
\\
+\mathbf{E}\bigl|g_2\bigl(\tilde X(I'_1),\ldots, \tilde X(I'_{s_2})\bigr)\bigr| \Bigl|\mathbf{E}g_1\bigl(X(I_1),\ldots, X(I_{s_1})\bigr) -
\mathbf{E}g_1\bigl(\tilde X(I_1),\ldots, \tilde X(I_{s_1})\bigr)\Bigr|\le\notag
\\
\le 4(|I|+|I'|)F' F'' \beta(m).\label{eq101c}
\end{gather}

The formulas~\eqref{g1g2}, \eqref{eq101c} and the triangle inequality give~\eqref{Lemma2*}. 
\end{proof}

\begin{proof}[Proof of Lemma~$\ref{l1}$] 
Lemma~\ref{l1} is immediate from Lemma~\ref{l2}, if we consider that $|f_{n;j_1,\ldots,j_r}(x_1,\ldots,x_r)|\le F_n$ and $|I|+|I'|\le |\V'|$.
Thus,~\eqref{Lemma2*} leads to~\eqref{Lemma1}. 
\end{proof}

\begin{proof}[Proof of Theorem~$\ref{th1}$]
We estimate the quantities in the condition~\eqref{eq:CondTh1} for our case. We begin with 
\begin{equation}\label{QRm}
Q_{R,n,m}=\underset{\V'\subset \V(\varGamma_{n,m}):|\V'| \leq R}{\max}\sum_{\tilde{\alpha}\in L(\V')}\mathbf{E}\left(|w_n({\tilde{\alpha}})|\bigl|\sigma\{w_n({\alpha}),\alpha\in \V'\}\right),
\end{equation}
where we recall $w_n({\alpha})=f_n(X_{n,j_1},\ldots, X_{n,j_r})$.

The set of strong dependencies $L(\alpha)=L_{n,m}(\alpha)$ for the vertex $\alpha=(j_1,\ldots,j_r)$ in the graph $\varGamma_{n,m}$ is the set of vertices $\tilde{\alpha}=(\tilde{j}_1,\ldots,\tilde{j}_r)\in \V(\varGamma_{n,m})$ for which at least one of the inequalities~\eqref{eq:CondV} holds. The number of the elements of this set satisfies the inequality 
\begin{equation}\label{L1}
|L(\alpha)|\le r^2(2m+1) C_n^{r-1}.
\end{equation}
The set of strong dependencies $L(\V')=L_{n,m}(\V')$ for the set $\V'\subset \V(\varGamma_{n,m})$ is given by the formula 
$$
L(\V') =\bigcup_{\alpha\in \V'}L(\V')
$$
and satisfies the inequality
\begin{equation}\label{|L|}
|L(V')| \le  r^2|\V'|(2m+1)C_n^{r-1}.
\end{equation}
From formulas~\eqref{|L|}, $|f_n(x,y)|\le F_n$ and~\eqref{QRm}, we have
\begin{equation}\label{QRm-1}
Q_{R,n,m} \le  r^2 R F_n (2m+1) C_n^{r-1}.
\end{equation}

We note that, in the case under consideration,
\begin{equation}\label{M}
M_n=\sum_{1\le j_1<\ldots <j_r \le n}\mathbf{E}|f_{n;j_1,\ldots,j_r}(X_{n,j_1},\ldots,X_{n,j_r})|\le F_nC_n^r.
\end{equation}

The use of~\eqref{Lemma1}, \eqref{QRm-1}, \eqref{M}, equality~\eqref{gamma}, and the  above cited Theorem~1 from~\cite{TihChis} leads to the following condition of the asymptotic normality for~$U_n$ in the triangular array scheme: for every natural number $R$ and all $b\le 2/3$
$$
\frac{F_n^2 m_n^{2-b}n^{2(r-1)+b}r^{4-2b}}{\mathbf{D}U_{n}}+\bigl(\beta_n(m_n)\bigr)^{b}\frac{F_n^{2} n^{2r}}{\mathbf{D}U_{n}}\rightarrow 0.
$$
\end{proof}

\begin{proof}[Proof of Theorem~$\ref{th2}$] Let the number~$F_n$ be independent of $n$, $\beta_n(m)\le m^{-h(m)}$, where $h(m)\to\infty$ ($m\to\infty$), and $\mathbf{D}U_n\ge Cn^{2(r-1)+\varkappa}$, where $C, \varkappa>0$. In this case~\eqref{eq:CondTheorem2} follows from the formula:
\begin{equation}\label{eq:CondTheorem2a}
m_n^{2-b}n^{b-\varkappa}+m_n^{-h(m_n)b} n^{2(r-1)-\varkappa}\rightarrow 0.
\end{equation}
Let us examine this. We put
$$
m_n = \left[n^{(\varkappa-b_0)/4}\right],\quad 0<b_0<\min\left\{2/3,\varkappa\right\},
$$
then, the expression in the left side~\eqref{eq:CondTheorem2a} for $0<b\le b_0$ and $n\to\infty$ can be estimated by
$O\left(n^{(b_0-\varkappa)/2}\right)\rightarrow 0$.
\end{proof}

\section{Remarks About Proofs of Theorems~\ref{th1a} and~\ref{th2a} for $V$-Statistics}

The proofs of Theorems~\ref{th1} and~\ref{th2} for $V$-statistics are completely the same, with the only difference  that the graph $\Gamma_{n, m} $ contains $n^r$ vertices, and in the formulas~\eqref {L1},~\eqref{|L|},~\eqref {QRm-1}, and~\eqref{M} the binomial coefficients $C_n^s $ must be replaced by $n^s,$ $s=r-1, r.$ The indicated replacement does not affect the form of the condition~\eqref{eq:CondTheorem2}.

For example, consider the formula~\eqref{L1}. Number of elements $\tilde{\alpha}=(\tilde{j}_1,\ldots,\tilde{j}_r): \tilde{j}_1,\ldots,\tilde{j}_r=1,\ldots,n,$ for which at least one of the inequalities~\eqref{eq:CondV} holds, can be estimated as follows. First, we select the indices $k$ and $l$,  there are $2m+1$ elements satisfying~\eqref{eq:CondV} for each such pair, and the rest of the elements can be any. Thus,
$$
|L(\alpha)|\le r^2 (2m+1)n^{r-1}.
$$
The remaining calculations are carried out similarly.

\end{document}